\documentclass[12pt]{article}
\usepackage{epsf}
\usepackage{amsbsy,amsmath}
\usepackage{mathtools}
\usepackage{mathrsfs}
\usepackage{amsfonts}
\usepackage{amssymb}
\usepackage{enumitem}
\usepackage{eucal}
\usepackage{graphics,mathrsfs}
\usepackage{amsthm}
\usepackage{secdot}
\newtheorem{theorem}{Theorem}[section]
\newtheorem{lemma}[theorem]{Lemma}

\theoremstyle{definition}
\newtheorem{definition}[theorem]{Definition}

\theoremstyle{remark}

\numberwithin{equation}{section}

\numberwithin{equation}{section}

\newsavebox{\savepar}

\begin{document}

\title{Problem involving nonlocal operator}
\author{Ratan Kr. Giri, D. Choudhuri~\footnote{Corresponding
author: dc.iit12@gmail.com}~~and Amita Soni \\
}
\date{}
\maketitle

\begin{abstract}
\noindent The aim of this paper is to deal with the elliptic pdes
involving a nonlinear integrodifferential operator, which are
possibly degenerate and covers the case of fractional $p$-Laplacian
operator. We prove the existence of a solution in the weak sense to
the problem
\begin{align*}
\begin{split}
-\mathscr{L}_\Phi u & = \lambda |u|^{q-2}u\,\,\mbox{in}\,\,\Omega,\\
u & = 0\,\, \mbox{in}\,\, \mathbb{R}^N\setminus \Omega
\end{split}
\end{align*}
if and only if a weak solution to
\begin{align*}
\begin{split}
-\mathscr{L}_\Phi u & = \lambda |u|^{q-2}u +f,\,\,\,f\in L^{p'}(\Omega),\\
u & = 0\,\, \mbox{on}\,\, \mathbb{R}^N\setminus \Omega
\end{split}
\end{align*}
($p'$ being the conjugate of $p$), exists in a weak sense, for
$q\in(p, p_s^*)$ under certain condition on $\lambda$, where
$-\mathscr{L}_\Phi $ is a general nonlocal integrodifferential
operator of order $s\in(0,1)$ and $p_s^*$ is the fractional Sobolev
conjugate of $p$. We further prove the existence of a measure
$\mu^{*}$ corresponding to which a weak solution exists to the
problem
\begin{align*}
\begin{split}
-\mathscr{L}_\Phi u & = \lambda |u|^{q-2}u +\mu^*\,\,\,\mbox{in}\,\, \Omega,\\
u & = 0\,\,\, \mbox{in}\,\,\mathbb{R}^N\setminus \Omega
\end{split}
\end{align*}
depending upon the capacity.
\\
{\bf keywords}:~Nonlocal operators, fractional $p$-laplacian; elliptic PDE; fractional Sobolev space.
\end{abstract}

\section{Introduction}
In the recent years, a great attention has been focused on the study
of fractional and nonlocal operators of elliptic type, both, for
research in pure Mathematics and for concrete real world
applications. From a physical point of view, the nonlocal operators
play a crucial role in describing several phenomena such as, the
thin obstacle problem, optimization, phase transitions, material
science, water waves, geophysical fluid dynamics and mathematical
finance. For further details on these applications, the reader may
refer to \cite{cur}, \cite{caf1}, \cite{caf2} and the references
therein. For a general reference to this topic one may refer to the
recent article of V\'azquez \cite{vaz}. Off late there has been a
rapid growth in the literature on problems involving nonlocal
operators. A testimony to this can be found in \cite{sor},
\cite{col}, \cite{molica1}, \cite{molica2}, \cite{yu}, \cite{su} in
the form of existence and multiplicity results for nonlocal
operators like fractional Laplacian, fractional $p$-Laplacian in
combination with a convex or a concave type non linearity. An
eigenvalue problem for the fractional $p$-Laplacian and properties
like finding the smallest eigenvalue are studied in \cite{inna},
\cite{eig1} and \cite{eig}. The Brezis-Nirenberg results for the
fractional Laplacian and for the fractional $p$-Laplacian operator
has been considered in \cite{serv3} and \cite{perera} respective. A
Dirichlet boundary value problem in the case of fractional Laplacian
with polynomial type nonlinearity using variational methods has been
studied in \cite{deb}, \cite{serv}, \cite{serv2}, \cite{cabre}. In
\cite{attar}, the authors proved the existence of weak solution on
the fractional $p$-Laplacian equations with weight for any datum in
$L^1$. Recently, Kussi et. al \cite{kussi}, has established an
existence, regularity and potential theory for a nonlocal
integrodifferential equations involving measure data.%operator
%featuring singular or very irregular data, as for those modeling
%source term that are concentrated at point.
The nonlocal elliptic operators considered here possibly degenerate
and cover the case of the fractional $p$-Laplacian operator. Based
on some generalization of the Wolff potential theory, the authors
obtained the existence of a weak solution belonging to a suitable fractional Sobolev space.\\
Motivated by the interest shared by the mathematical community in
this topic,  we study here the equivalence of the following two
problems involving a nonlocal operator,
\begin{align}
\begin{split}
P_1:~~-\mathscr{L}_\Phi u & = \lambda |u|^{q-2}u\,\,\mbox{in}\,\,\Omega,\\
u & = 0\,\, \mbox{on}\,\, \mathbb{R}^N\setminus \Omega\label{e1}
\end{split}
\end{align}
and
\begin{align}
\begin{split}
P_2:~-\mathscr{L}_\Phi u & = \lambda |u|^{q-2}u +f,\,\,\,f\in L^{p'}(\Omega),\\
u & = 0\,\, \mbox{on}\,\, \mathbb{R}^N\setminus \Omega\label{e2}
\end{split}
\end{align}
in the sense that if one problem has a nontrivial weak solution then
the other one also has a nontrivial weak solution. Here $\Omega$ is
a bounded open subset of $\mathbb{R}^N$ for $N\geq2$ and
$-\mathscr{L}_\Phi$ is a nonlocal operator (refer, \cite{kussi})
which is defined as
\begin{equation}
<-\mathscr{L}_\Phi u, \varphi>
=\int_{\mathbb{R}^N}\int_{\mathbb{R}^N}
\Phi(u(x)-u(y))(\varphi(x)-\varphi(y))K(x,y)dxdy,\label{e3}
\end{equation}
for every smooth function $\varphi$ with compact support, i.e.,
$\varphi \in C_c^\infty(\mathbb{R}^N)$. The function $\Phi$ is a
real valued continuous function over $\mathbb{R}$, satisfying
$\Phi(0)=0$ together with the following monotonicity property
\begin{equation}
\Lambda ^{-1} |t|^p \leq \Phi(t)t\leq \Lambda |t|^p, \,\,\forall
\,t\in \mathbb{R}.\label{e4}
\end{equation}
The kernel $K : \mathbb{R}^N \times \mathbb{R}^N \rightarrow
\mathbb{R}$ is a measurable function satisfying the following
ellipticity property
\begin{equation}
\frac{1}{\Lambda |x-y|^{N+sp}} \leq K(x,y) \leq
\frac{\Lambda}{|x-y|^{N+sp}}, \,\forall x, y \in \mathbb{R}^N, x\neq
y,\label{e5}
\end{equation}
where $\Lambda \geq 1$, $s\in (0,1), p > 2- \frac{s}{N}$ with $
q\in(p, \frac{Np}{N-sp}= p_s^{*})$ and $p'= \frac{p}{p-1}$, the
conjugate of $p$. Assumptions made in $(\ref{e4})$ and $(\ref{e5})$
makes the nonlocal operator $-\mathscr{L}_\Phi$ to be an elliptic
operator. Note that, upon taking the special case $\Phi(t)=
|t|^{p-2}t$ with $K(x,y)= |x-y|^{-(N+sp)}$ in $(\ref{e3})$, we
recover the fractional
$p$-Laplacian, defined by $(-\Delta)_p^su=\int_{\Omega}\int_{\Omega}\frac{|u(x)-u(y)|^p}{|x-y|^{N+ps}}dxdy$. %On the other hand,
%by taking $\phi(t)= t$ with the same kernel $|x-y|^{-(N+sp)}$, one
%can also recover the case of the classical fractional Laplacian,
%$\mathscr{L}_\phi=(-\Delta)^s$.
In other words, the nonlocal operator $-\mathscr{L}_\Phi$, is a
generalization of the fractional $p$-Laplacian for $1 \leq p < \infty$ and $s\in (0,1)$.\\
A similar type of equivalence result for problems $P_1$ and $P_2$,
but involving the local operator $-\Delta_p$, defined as
$\nabla\cdot(|\nabla (.)|^{p-2}\nabla (.))$, in place of nonlocal
operator $-\mathscr{L}_\Phi$, are obtained in the work of Giri and
Choudhuri \cite{giri}. In this paper,  we use the variational
approach to the non-local framework. Inspired by the fractional
Sobolev spaces, we will work in a functional analytic set up, in
order to correctly encode the Dirichlet boundary datum in the
variational formulation. We also characterize the measures
corresponding to which the problem
\begin{align}
\begin{split}
-\mathscr{L}_\Phi u & = \lambda |u|^{q-2}u+\mu^*\,\,\mbox{in}\,\,\Omega,\nonumber\\
u & = 0\,\, \mbox{inn}\,\, \mathbb{R}^N\setminus \Omega\nonumber
\end{split}
\end{align}
has a weak solution using the notion of {\it capacity}.
 The paper is organized as
follows. In section $2$, we present some useful tools and
preliminaries that we will use through this paper, like fractional
Sobolev spaces $W^{s, p}(\mathbb{R^N})$ and embedding results of
$W^{s, p}(\mathbb{R^N})$ into the Lebesgue spaces. We also define
the weak sense in which the solutions to the problems $P_1$ and
$P_2$ are defined. In section $3$, we discuss a few preliminary
results and the main result.
\section{Functional analytic setup and the Main tools}
In this section, we discuss the functional analytic setting that
will be used below. Due to the nonlocal character of
$\mathscr{L}_\Phi$ defined in $\ref{e3}$, it is natural to work with
Sobolev space $W^{s, p}(\mathbb{R^N})$ and express the Dirichlet
condition on $\mathbb{R}^N\setminus \Omega$ rather than
$\partial\Omega$. Though fractional Sobolev spaces are well known
since the beginning of the last century, especially in the field of
harmonic analysis, they have become increasingly popular in the last
few years, under the impulse of the work of Caffarelli \& Silvestre
\cite{caf2} and the references there in. We now turn to our problem
for which we provide the variational setting on a suitable function
space for $(\ref{e1})$ and $(\ref{e2})$, jointly with some
preliminary results. For all measurable functions $u:
\mathbb{R}^N\rightarrow \mathbb{R}$, we set
$$||u||_{L^p(\mathbb{R}^N)}=\left(\int_{\mathbb{R}^N}|u(x)|^p dx\right)^{\frac{1}{p}},$$
$$[u]_{s,p}= \left(\int_{\mathbb{R}^{2N}}\frac{|u(x)-u(y)|^p}{|x-y|^{N+sp}}
dxdy\right)^{\frac{1}{p}},$$ where $p\in(1, \infty)$ and $s\in
(0,1)$. The fractional Sobolev space $W^{s,p}(\mathbb{R}^N)$ is
defined as the space of all function $u\in L^p(\mathbb{R}^N)$ such
that $[u]_{s,p}$ is finite and endowed with the norm
$$||u||_{W^{s,p}(\mathbb{R}^N)}=
\left(||u||_{L^p(\mathbb{R}^N)}^p+[u]_{s,p}^p\right)^{\frac{1}{p}}.$$
More on fractional Sobolev space can be found in Nezza et al.
\cite{nezza} and the references therein. We now define a closed
linear
subspace of $W^{s,p}(\mathbb{R}^N)$: \\
$$W_0^{s,p}(\Omega)= \{u\in W^{s,p}(\mathbb{R}^N): u=0\,\,
a.e.\,\,\text{in}\,\,\mathbb{R}^N\setminus\Omega\}$$. It is trivial
to see that the norms $||\cdot||_{L^p(\mathbb{R}^N)}$ and $||\cdot
||_{L^p(\Omega)}$ agree on $W_0^{s,p}(\Omega)$. We also have a
Poincar\'{e} type inequality which is as follows.
$$||u||_{L^p(\Omega)}\leq C
[u]_{s,p},\,\,\,\mbox{for\,\,all}\,\,u\in W_0^{s,p}(\Omega).$$ Thus,
we can equivalently renorm $W_0^{s,p}(\Omega)$, by setting
$||u||_{W_0^{s,p}(\Omega)}= [u]_{s,p}$, for every $u\in
W_0^{s,p}(\Omega)$. Let $p_s^*= \frac{Np}{N-sp}$ , with the
agreement that $p_s^*=\infty$ if $N\geq sp$. It is well known that
$(W_0^{s,p}(\Omega), ||\cdot||_{W_0^{s,p}(\Omega)})$ is a uniformly
convex reflexive Banach space, continuously embedded into
$L^r(\Omega)$, for all $r\in[1, p_s^*]$ if $sp<N$, for all $1\leq r
<\infty$ if $N=sp$ and into $L^\infty(\Omega)$ if $N<sp$. It is also
compactly embedded in $L^r(\Omega)$ for any $r\in[1, p_s^*)$ if
$N\geq sp$ and in $L^\infty(\Omega)$ for $N<sp$. Furthermore,
$C_c^\infty(\Omega)$ is a dense subspace of $W_0^{s,p}(\Omega)$ with
respect to $||\cdot||_{W_0^{s,p}(\Omega)}$. For further detail on
the embedding results, we refer the reader to \cite{sre1},
\cite{sre2} and
the references there in.\\
We define an associated energy functional to the problem $P_1$ as
$$I_{P_1}(u) =
\int_{\mathbb{R}^N}\int_{\mathbb{R}^N}
\mathcal{P}\Phi(u(x)-u(y))K(x,y)dx dy -\frac{\lambda}{q}\int_\Omega
|u|^q dx,$$ where $\mathcal{P}\phi(t) :=
\int_0^{|t|}\phi(\tau)d\tau$ being the primitive of $\Phi$. Thus by
$(\ref{e4})$ we have
\begin{equation} \Lambda^{-1}\frac{|t|^p}{p}\leq
\mathcal{P}\Phi(t)\leq \Lambda \frac{|t|^p}{p},\label{e6}
\end{equation}
for $t\neq 0$ and $\mathcal{P}\Phi(0) =0$. The Fr\'{e}chet
derivative of $I$, which is in $W_0^{-s, p'}(\Omega)$, the dual
space of $W_0^{s, p}(\Omega)$ where $p'= \frac{p}{p-1}$ and is
defined as
\begin{equation}
<I'(u), v> = \int_{\mathbb{R}^N}\int_{\mathbb{R}^N}
\Phi(u(x)-u(y))(v(x)-v(y))K(x,y)dxdy -\lambda\int_\Omega |u|^{q-2}u
vdx,\label{e8'}.
\end{equation}
\begin{definition}{\label{df1}}
We say that $u\in W_0^{s,p}(\Omega)$ is a weak (energy) solution to
the problem $P_1$ if
$$\int_{\mathbb{R}^N}\int_{\mathbb{R}^N}\Phi(u(x)-u(y))(\varphi(x)-\varphi(y))K(x,y)dxdy
= \lambda\int_\Omega |u|^{q-2}u\varphi dx ,$$ holds for every
$\varphi\in C_c^\infty(\Omega)$.
\end{definition}
% will find solutions of problems $P_1$ and $P_2$ in the closed
%subspace $W_0^{s,p}(\Omega)$.\\
\noindent The weak solutions of the problem $P_1$ are the critical
points of the energy functional $I_{P_1}$.
%We now consider the
%non-homogeneous counterpart of the first problem $p_1$ -which is the
%the second problem $P_2$.
\noindent Similarly, let the corresponding associated energy
functional to the problem $P_2$ be denoted by $I_{P_2}$ which is
defined as follows.
\begin{equation}
I_{P_2}(u)= \int_{\mathbb{R}^N}\int_{\mathbb{R}^N}
\mathcal{P}\Phi(u(x)-u(y))K(x,y)dx dy -\frac{\lambda}{q}\int_\Omega
|u|^q dx - \int_\Omega f udx\label{e7}
\end{equation}
whose Fr\'{e}chet derivative is defined as
\begin{equation}
<I'(u), v> = \int_{\mathbb{R}^N}\int_{\mathbb{R}^N}
\Phi(u(x)-u(y))(v(x)-v(y))K(x,y)dxdy -\lambda\int_\Omega |u|^{q-2}u
vdx-\int_\Omega fvdx,\label{e8}
\end{equation}
for every $v \in W_0^{s, p}(\Omega)$.
\begin{definition}
$u\in W_0^{s, p}(\Omega)$ is a weak (energy) solution of the problem
$P_2$ if
\begin{equation*}
\int_{\mathbb{R}^N}\int_{\mathbb{R}^N}\Phi(u(x)-u(y))(\varphi(x)-\varphi(y))K(x,y)dxdy
- \lambda\int_\Omega |u|^{q-2}u\varphi dx - \int_\Omega f\varphi dx
=0,
\end{equation*}
 for all $\varphi \in C_c^\infty(\Omega)$.
 \end{definition}
The main results in this paper, when stated heuristically, are as
follows. The problem $P_1$ has a nontrivial weak solution if and
only if the problem $P_2$ has a non trivial weak solution. In the
implication from $P_1$ to $P_2$, the main tools we will use the
Mountain-Pass Theorem \cite{evans,kesavan}. For the converse part,
we guarantee the existence of a weak solution to the problem $P_1$
by considering a sequence $P_2$ type problems whose nonhomogeneous
part will be denoted by $f_n$. In addition to this we will assume
that $f_n \rightarrow 0$ in $L^{p'}(\Omega)$. The corresponding
sequence of weak solutions $(u_n)$ to the problem $P_2$, which will
be shown to have a strongly convergent subsequence. We complete the
proof of the converse part by showing that this limit $u$ is a weak
solution to $P_1$. One common result, which will be used in proving
both the implications is as follows.
% and the limit of this subsequence will be the weak
%solution.
\begin{theorem}
If the sequence $u_n \rightharpoonup u$ in $W_0^{s, p}(\Omega)$,
then $$\int_{\mathbb{R}^N}\int_{\mathbb{R}^N}
(\Phi(u_n(x)-u_n(y))-\Phi(u(x)-u(y)))(\varphi(x)-\varphi(y))K(x,y)dxdy
\rightarrow 0$$ for all $\varphi\in C_c^\infty(\Omega)$.
\end{theorem}
\noindent Proof of this theorem, can be found in the work of Kussi
\cite{kussi}, Theorem 1.1.
%Now we turn to our existence results.
\section{Existence Results}
We bgin the section by assuming that the problem $P_1$ has a
nontrivial weak solution in $W_0^{s, p}(\Omega)$. In order to show
the existence of a non trivial weak solution to the problem $P_2$,
we will use the Mountain-pass theorem. To apply the Mountain-pass
theorem, we need the following technical lemmas.
\begin{lemma}\label{lem1}
The function $I$ defined in ($\ref{e7}$) is a $C^1$ functional over
$W_0^{s, p}(\Omega)$.
\end{lemma}
\begin{proof}
It is trivial to see that the functional $I$ is differentiable over
$W_0^{s,p}(\Omega)$. Thus it is enough to show that that $I'(u)$ is
continuous. Thus from ($\ref{e8}$), for each $u\in
W_0^{s,p}(\Omega)$ we have
\begin{align*}
|<I'(u), v>| & \leq \int_{\mathbb{R}^N}\int_{\mathbb{R}^N}
|\phi(u(x)-u(y))||(v(x)-v(y))|K(x,y)dxdy + \lambda \int_\Omega
|u|^{q-1}|v|dx + \int_\Omega |f||v|dx\\
& \leq \Lambda^2 \int_{\mathbb{R}^N}\int_{\mathbb{R}^N}
\frac{|u(x)-u(y)|^{p-1}|v(x)-v(y)|}{|x-y|^{N+sp}}dxdy + \lambda
||u||_{\frac{q}{q-1}}||v||_q +||f||_{p'}||v||_p\\
& \leq \left[
\left|\left|\frac{|u(x)-u(y)|^{p-1}}{|x-y|^{\frac{N+sp}{p'}}}\right|\right|_{p'}+
C_1 \lambda ||u||_{\frac{q}{q-1}}+ C_2||f||_{p'}
\right]||v||_{W_0^{s, p}(\Omega)},
\end{align*}
where $C_1, C_2$ are constant due to the embedding of $W_0^{s,
p}(\Omega)$ into $L^q(\Omega)$ for $q\in(p, p_s^*)$ and into
$L^p(\Omega)$ respectively. Thus $I$ is a $C^1$ functional over
$W_0^{s, p}(\Omega)$.
\end{proof}
\begin{lemma}\label{lem2}
There exists $u_0, u_1\in W_0^{s, p}(\Omega)$ and a positive real
number $C_3$ such that $I(u_0), I(u_1)<C_3$ and $I(v)\geq C_3$, for
every $v$ satisfying $||v-u||_{W_0^{s, p}(\Omega)}=r$ for some
$r>0$.
\end{lemma}
\begin{proof}
Let us consider $u_0=0$ and let $w\in B(0,1) =\{u\in W_0^{s,
p}(\Omega): ||u||_{W_0^{s, p}(\Omega)}=1\}$. Consider $v=u_0 +rw$
for $r>0$ so that $||v-u_0||_{W_0^{s, p}(\Omega)}=r$. We first show
that there exists a $r>0$ such that for each $v$ satisfying
$||v-u_0||_{W_0^{s, p}(\Omega)}=r$ we have $I(v)\geq C_3$, where
$C_3>0$. From the monotonicity and ellipticity conditions in
$(\ref{e4})$ and $(\ref{e5})$ respectively
\begin{align*}
I(u)& =\int_{\mathbb{R}^N}\int_{\mathbb{R}^N}
\mathcal{P}\Phi(u(x)-u(y))K(x,y)dx dy -\frac{\lambda}{q}\int_\Omega
|u|^q dx - \int_\Omega f udx\\
& \geq \frac{\Lambda^{-2}}{p}||u||^p_{W_0^{s, p}(\Omega)} -
\frac{\lambda}{q}\int_\Omega |u|^q dx - \int_\Omega f udx
\end{align*}
Thus we have,
$$I(u_0+rw)-I(u_0) \geq \frac{\lambda^{-2}r^p}{p}||w||^p_{W_0^{s,
p}(\Omega)}-\frac{\lambda r^q}{q}||w||_q^q -r\int_\Omega fwdx.$$
Note that $||w||_{W_0^{s, p}(\Omega)}=1$ and $W_0^{s,
p}(\Omega)\hookrightarrow L^p(\Omega)$, hence $||w||^p_p\leq
C_4||w||_{W_0^{s, p}(\Omega)}=C_4$. Similarly, since $W_0^{s,
p}(\Omega)\hookrightarrow L^q(\Omega)$ for $q\in(p,p_s^*)$, hence we
have $||w||^q_q \leq C_5$. This leads to
$$I(u_0+rw)-I(u_0) \geq r\left[ \frac{\Lambda^{-2}r^{p-1}}{p}-\frac{r^{q-1}\lambda}{q}C_5 -C_4^{\frac{1}{p}}||f||_p'
\right].$$ Let
$F(r)=\frac{\Lambda^{-2}r^{p-1}}{p}-\frac{r^{q-1}\lambda}{q}C_5
-C_4^{\frac{1}{p}}||f||_p'$. Then $F(0)<0$ and for $r_0 =
\frac{q(p-1)}{p(q-1)}\cdot\frac{\Lambda^{-2}}{\lambda C_5}$, we see
that $F'(r_0)=0$. Further, a bit of calculus guarantees that
$F''(r_0)<0$ and hence $r_0$ is a maximizer of $F$. Note that, if
$0<\lambda < \lambda_1 =
\frac{\Lambda^{-2}q(p-1)}{C_5p(q-1)}\left(\frac{q-p}{p(q-1)}\cdot
\frac{1}{C_4^{1/p}||f||_{p'}}\right)^{\frac{q-p}{p-1}}$, then
$F(r_0)>0$. Hence there exists $r_1, r_2>0$ and $r_1<r_0<r_2$ such
that $F(r)>0$ for each $r\in(r_1, r_2)$. Thus, we choose $r=0$ such
that $||v-u_0||_{W_0^{s, p}(\Omega)}=r_0$ and for which $I(v)\geq
C_3>0$ for each $v$ such that $||v-u_0||_{W_0^{s,p}(\Omega)}=r_0$.\\
{\it Choice of $u_1$}: Let $w_q$ be a nontrivial solution of the
problem $P_1$. Then consider $g=kw_q$, $k\in \mathbb{R}$, where we
have normalized $w_q$ with respect to the norm of $W_0^{s,
p}(\Omega)$ without changing its notation. Now we have,
$$I(u)\leq \frac{\Lambda^{2}}{p}||u||^p_{W_0^{s, p}(\Omega)} -
\frac{\lambda}{q}\int_\Omega |u|^q dx - \int_\Omega f udx.$$ From
this it can be seen that
$$I(g)\leq \frac{\Lambda^{2}k^p}{p}-\frac{\lambda
k^q}{q}\int_\Omega|w_q|^q dx -kC,$$ where $C=\int_\Omega fw_qdx$.
Since $p<q<p_s^*$, we observe that $k$ can be chosen so large that
$\frac{\Lambda^{2}k^p_0}{p}-\frac{\lambda
k^q_0}{q}\int_\Omega|w_q|^q dx -k_0C<0$. Then $I(k_0w_q) <0$. Thus
we can choose $u_1= k_0w_q$, where $k_0>r_0$, due to which
$||u_1-u_0||_{W_0^{s, p}(\Omega)}>r_0$. Hence the lemma follows.
\end{proof}
\begin{lemma}\label{lem3}
The functional $I$ satisfies the Palais-Smale condition, when
$\Lambda \in [1, (\frac{q}{p})^{1/4})$.
\end{lemma}
\begin{proof}
Let $(u_m)$ be a sequence in $W_0^{s, p}(\Omega)$ such that
$|I(u_m)|\leq M$ and $I'(u_m)\rightarrow 0$ in $W^{-s, p'}(\Omega)$.
Now,
$$<I'(u_m), v> = \int_{\mathbb{R}^N}\int_{\mathbb{R}^N}
\Phi(u_m(x)-u_m(y))(v(x)-v(y))K(x,y)dxdy -\lambda\int_\Omega
|u_m|^{q-2}u_m vdx-\int_\Omega fvdx,$$ for every $v\in W_0^{s,
p}(\Omega)$. From the definition of the functional and its
derivative we have
$$<I'(u_m), u_m>=\int_{\mathbb{R}^N}\int_{\mathbb{R}^N}
\Phi(u_m(x)-u_m(y))(u_m(x)-u_m(y))K(x,y)dxdy -\lambda\int_\Omega
|u_m|^qdx-\int_\Omega fu_mdx$$
$$I(u_m)=
\int_{\mathbb{R}^N}\int_{\mathbb{R}^N}\mathcal{P}\Phi(u_m(x)-u_m(y))K(x,y)dxdy-\frac{\lambda}{q}\int_\Omega
|u_m|^q dx-\int_\Omega fu_mdx.$$ From the above two equations it
follows that
\begin{align}
q\int_{\mathbb{R}^N}\int_{\mathbb{R}^N}\mathcal{P}\Phi(u_m(x)-u_m(y))
& K(x,y)dxdy -
\int_{\mathbb{R}^N}\int_{\mathbb{R}^N}\Phi(u_m(x)-u_m(y))(u_m(x)-u_m(y))K(x,y)dxdy \nonumber \\
& =  q I(u_m) -<I'(u_m), u_m> + (q-1)\int_\Omega f u_mdx. \label{e9}
\end{align}
From ($\ref{e4}$), we have
$$ \frac{1}{\Lambda^{2}}||u_m||^p_{W_0^{s,p}(\Omega)}\leq
\int_{\mathbb{R}^N}\int_{\mathbb{R}^N}\phi(u_m(x)-u_m(y))(u_m(x)-u_m(y))K(x,y)dxdy\leq
\Lambda^{2}||u_m||^p_{W_0^{s, p}(\Omega)}$$ and hence
$$\frac{1}{\Lambda^2p}||u_m||^p_{W_0^{s, p}(\Omega)}\leq  \int_{\mathbb{R}^N}\int_{\mathbb{R}^N}\mathcal{P}\phi(u_m(x)-u_m(y))K(x,y)dxdy \leq
\frac{\Lambda^2}{p}||u_m||^p_{W_0^{s, p}(\Omega)}.$$  From
($\ref{e9}$), we get
\begin{align*}
\frac{(q-p\Lambda^4)}{\Lambda^2 p}||u_m||^p_{W_0^{s, p}(\Omega)} & \leq q I(u_m) -<I'(u_m), u_m> + (q-1)\int_\Omega f u_mdx\\
& \leq q M + ||I'(u_m)||_{-s,p'}||u_m||_{W_0^{s, p}(\Omega)} +
(q-1)c||f||_{p'}||u_m||_{W_0^{s, p}(\Omega)}
\end{align*}
It follows from this that $(||u_m||_{W_0^{s, p}(\Omega)})$ is
bounded. If not then, divide by $||u_m||^p_{W_0^{s, p}(\Omega)}$ in
the above and let $m\rightarrow \infty$. On using
$||I'(u_m)||_{-s,p'}\rightarrow 0$, we get a contradiction, viz.,
$\frac{(q-p\Lambda^4)}{\Lambda^2 p}\leq 0$, by the assumption $q-p\Lambda^4 >0$. \\
Thus there exists a subsequence of $(u_m)$, which will still be
denoted by $(u_m)$, converge weakly to $u$ in $W_0^{s, p}(\Omega)$.
We will now show that this subsequence $(u_m)$ is
strongly convergent in $W_0^{s, p}(\Omega)$. \\
We know that $W_0^{s, p}(\Omega)$ is compactly embedded in
$L^r(\Omega)$, $r\in[1,p_s^*)$ and hence $u_m\rightarrow u$ in
$L^r(\Omega)$, for $1\leq r <p_s^*$. Consider $\tilde{u}_m = u_m-u$.
Then $ \tilde{u}_m \rightharpoonup 0$ in $W_0^{s, p}(\Omega)$.
Consider,
\begin{eqnarray}
<I'(\tilde{u}_m),\tilde{u}_m>
&=&\int_{\mathbb{R}^N}\int_{\mathbb{R}^N}\Phi(\tilde{u}_m(x)-\tilde{u}_m(y))(\tilde{u}_m(x)-\tilde{u}_m(y))K(x,y)dxdy
\nonumber\\
&-&\lambda\int_\Omega |\tilde{u}_m|^qdx-\int_\Omega
f\tilde{u}_mdx\label{e10}
\end{eqnarray}
The second and the third term of the functional approach to $0$ as
$m\rightarrow\infty$. This is because $\tilde{u}_m \rightharpoonup
0$ in $W_0^{s, p}(\Omega)$ implies $\int_\Omega f\tilde{u}_mdx
\rightarrow 0$ and $u_m\rightarrow u$ in $L^r(\Omega)$ for $r\in
[1,p_s^*)$ implies $\int_\Omega |\tilde{u}_m|^qdx \rightarrow 0$. By
the definition of weak convergence, for all $u\in W_0^{s,
p}(\Omega)$, $<I'(u), \tilde{u}_m> \rightarrow 0$, as $m\rightarrow
\infty$. That is, $\displaystyle{\lim_{m\rightarrow \infty} <I'(u),
\tilde{u}_m> =0}$, for all $u\in W_0^{s, p}(\Omega)$. Upon taking
$u=\tilde{u}_n$,
$$\lim_{n\rightarrow \infty}\lim_{m\rightarrow \infty} <I'(\tilde{u}_n), \tilde{u}_m>
=0.$$ We now consider,
\begin{equation}
<I'(\tilde{u}_n),
\tilde{u}_m>=\int_{\mathbb{R}^N}\int_{\mathbb{R}^N}
\phi(\tilde{u}_n(x)-\tilde{u}_n(y))(\tilde{u}_m(x)-\tilde{u}_m(y))K(x,y)dxdy
-\lambda\int_\Omega |\tilde{u}_n|^{q-2}\tilde{u}_n
\tilde{u}_mdx-\int_\Omega f\tilde{u}_m dx.\label{e11}
\end{equation}
Since $\tilde{u}_n \rightharpoonup 0$, hence by the work of Kussi
\cite{kussi}, we have
$$\displaystyle{\lim_{n\rightarrow
\infty}\int_{\mathbb{R}^N}\int_{\mathbb{R}^N}
\phi(\tilde{u}_n(x)-\tilde{u}_n(y))(\tilde{u}_m(x)-\tilde{u}_m(y))K(x,y)dxdy=0.}$$
In addition to this $\displaystyle{\lim_{n\rightarrow
\infty}\int_\Omega |\tilde{u}_n|^{q-2}\tilde{u}_n \tilde{u}_mdx=
0}$. Therefore, taking limit $n\rightarrow \infty$ in $(\ref{e11})$,
we have,
$$\lim_{n\rightarrow \infty}<I'(\tilde{u}_n), \tilde{u}_m> =
\int_\Omega f\tilde{u}_m dx.$$ On further taking limit $m\rightarrow
\infty$ in the above we get,
$$\lim_{m\rightarrow \infty}\lim_{n\rightarrow \infty}<I'(\tilde{u}_n),
\tilde{u}_m> =0.$$ Therefore, we have
$$\lim_{m\rightarrow \infty}\lim_{n\rightarrow \infty}<I'(\tilde{u}_n),
\tilde{u}_m> = \lim_{n\rightarrow \infty}\lim_{m\rightarrow \infty}
<I'(\tilde{u}_n), \tilde{u}_m> =0.$$ Hence it follows that
$\underset{m\rightarrow \infty}{\lim} <I'(\tilde{u}_m), \tilde{u}_m>
=0$. Now taking limit $m\rightarrow \infty$ in ($\ref{e10}$), we get
$$\lim_{m\rightarrow \infty}\int_{\mathbb{R}^N}\int_{\mathbb{R}^N}
\Phi(\tilde{u}_m(x)-\tilde{u}_m(y))(\tilde{u}_m(x)-\tilde{u}_m(y))K(x,y)dxdy=0.$$
But,
$$ 0 \leq \frac{1}{\Lambda^{2}}||\tilde{u}_m||^p_{W_0^{s,
p}(\Omega)}\leq
\int_{\mathbb{R}^N}\int_{\mathbb{R}^N}\Phi(\tilde{u}_m(x)-\tilde{u}_m(y))(\tilde{u}_m(x)-\tilde{u}_m(y))K(x,y)dxdy.$$
From this it is clear that $\displaystyle{\lim_{m\rightarrow
\infty}||\tilde{u}_m||_{W_0^{s, p}(\Omega)}}=0$, that is
$\tilde{u}_m \rightarrow 0$ in $W_0^{s, p}(\Omega)$. Hence
$u_m\rightarrow u$ in $W_0^{s, p}(\Omega)$.
\end{proof}
\begin{theorem}
Suppose that the problem
\begin{align*}
\begin{split}
P_1:~~-\mathscr{L}_\Phi u & = \lambda |u|^{q-2}u\,\,\mbox{in}\,\,\Omega,\\
u & = 0\,\, \mbox{on}\,\, \mathbb{R}^N\setminus \Omega
\end{split}
\end{align*}
has a nontrivial weak solution for some $\lambda>0$, where $q\in(p,
p_s^*)$. Then there exists $\lambda_1>0$ such that for all
$\lambda\in (0,\lambda_1$), the problem
\begin{align*}
\begin{split}
P_2:~~-\mathscr{L}_\Phi u & = \lambda |u|^{q-2}u +f,\,\,\,f\in L^{p'}(\Omega),\\
u & = 0\,\, \mbox{on}\,\, \mathbb{R}^N\setminus \Omega
\end{split}
\end{align*}
has a nontrivial weak solution whenever $\Lambda \in[1,
(\frac{q}{p})^{1/4})$, where the $\Lambda$ given in the equation
($\ref{e4}$).
\end{theorem}
\begin{proof}
By the results proved in Lemmas $\ref{lem1}$, $\ref{lem2}$ and
$\ref{lem3}$, it follows that the functional $I$ associated with the
problem $P_2$, satisfies all the condition of Mountain-Pass theorem.
So by the Mountain-Pass theorem, an extreme point for $I$ exists in
$W_0^{s, p}(\Omega)$, which is a weak solution to the problem $P_2$.
\end{proof}
\noindent Conversely, suppose that  to each $f\in L^{p'}(\Omega)$,
the problem
\begin{align*}
\begin{split}
P_2:~~-\mathscr{L}_\Phi u & = \lambda |u|^{q-2}u +f,\,\,\,f\in L^{p'}(\Omega),\\
u & = 0\,\, \mbox{on}\,\, \mathbb{R}^N\setminus \Omega
\end{split}
\end{align*}
has a nontrivial solution on the set $\mathfrak{M}= \{u\in W_0^{s,
p}(\Omega): ||u||_q=1\}$ for some $\lambda >0$, where $q\in(p,
p_s^*)$. Existence of such a solution can be seen from the weak
lower semi continuity and coercivity of the associated functional
$I$ on the subset $\mathfrak{M}$ of $W_0^{s, p}(\Omega)$. In order
to show the existence of the problem $P_1$ for $q\in(p, p_s^*)$, let
us consider a sequence $(f_n)\subset L^{p'}(\Omega)$ such that
$f_n\rightarrow 0$ in $L^{p'}(\Omega)$. Then for each such $f_n$,
there exists a weak solution to the problem $P_2$, say $u_n$. Thus
each $u_n$ is a critical point of the functional $I$, i.e.,
$<I'(u_n), \varphi>=0$ for every $\phi\in W_0^{s, p}(\Omega)$. In
particular, $<I'(u_n), u_n>=0$. This implies that
$$\int_{\mathbb{R}^N}\int_{\mathbb{R}^N}
\Phi(u_n(x)-u_n(y))(u_n(x)-u_n(y))K(x,y)dxdy -\lambda\int_\Omega
|u_n|^q dx =\int_\Omega f_n u_n dx.\ref{eq11'}$$ Further, since
$||u_n||_q =1$ and using $(\ref{e4})$ in $(\ref{eq11'})$, we have
\begin{align*}
\Lambda^{-2}||u_n||^p_{W_0^{s, p}(\Omega)}-\lambda & \leq
\int_\Omega f_n u_n
dx\\
& \leq C ||f_n||_{p'}||u_n||_{W_0^{s, p}(\Omega)}\\
& \leq CM||u_n||_{W_0^{s, p}(\Omega)},
\end{align*}
Where $M$ is such that $||f_n||_{p'} \leq M$. It follows that
$(||u_n||_{W_0^{s, p}(\Omega)})$ is bounded, for if not, then divide
by $||u_n||^p_{W_0^{s, p}(\Omega)}$ and let $n\rightarrow \infty$,
to get a contradiction, viz., $\Lambda^{-2}\leq 0$. Hence there
exists a subsequence $(u_n)$, which converge weakly to $u$ in
$W_0^{s, p}(\Omega)$. Since $W_0^{s, p}(\Omega)$ is compactly
embedded in $L^q(\Omega)$ for $q\in(p, p_s^*)$, hence
$u_n\rightarrow u$ in $L^q(\Omega)$. Therefore,
$$\int_\Omega |u_n|^{q-2}u_n vdx \rightarrow \int_\Omega |u|^{q-2}u
vdx$$ and $\int_{\Omega}f_nvdx\rightarrow 0$ for each $v\in W_0^{s,
p}(\Omega)$. By kussi et al. \cite{kussi} we have,
$$\int_{\mathbb{R}^N}\int_{\mathbb{R}^N}
\Phi(u_n(x)-u_n(y))(v(x)-v(y))K(x,y)dxdy \rightarrow
\int_{\mathbb{R}^N}\int_{\mathbb{R}^N}
\Phi(u(x)-u(y))(v(x)-v(y))K(x,y)dxdy.\label{eq12'}$$
%Also  $\int_\Omega f_n v dx
%\rightarrow 0$ as $n\rightarrow 0$. Now for all $v\in W_0^{s,
%p}(\Omega)$,  we have,
%$$\int_{\mathbb{R}^N}\int_{\mathbb{R}^N}
%\phi(u_n(x)-u_n(y))(v(x)-v(y))K(x,y)dxdy -\lambda \int_\Omega
%|u_n|^{q-2}u_n vdx = \int_\Omega f_n v dx.$$
Taking limit $n\rightarrow 0$ in $(\ref{eq12'})$ we get
$$\int_{\mathbb{R}^N}\int_{\mathbb{R}^N}
\Phi(u(x)-u(y))(v(x)-v(y))K(x,y)dxdy - \int_\Omega |u|^{q-2}u vdx
=0,$$ for each $v\in W_0^{s, p}(\Omega)$. This shows that $u$ is
weak solution of the problem $P_1$.\\
Assume that $\lambda \in \left( 0, \displaystyle{\inf _{u\neq 0\in
W_0^{s, p}(\Omega)}\frac{||u||_{W_0^{s,
p}(\Omega)}}{||u||_q}}\right]$. As in the proof of Lemma
$\ref{lem3}$, it is easy to see that $$||u_n||_{W_0^{s, p}(\Omega)}
\rightarrow ||u||_{W_0^{s, p}(\Omega)}.$$ Since $||u_n||_q =1 $ and
$u_n\rightarrow u$ in $W_0^{s, p}(\Omega)$, we have
\begin{align*}
0<\lambda & \leq \lim \inf \frac{||u_n||_{W_0^{s,
p}(\Omega)}}{||u_n||_q}\\
& = \lim \inf ||u_n||_{W_0^{s, p}(\Omega)}\\
& = ||u||_{W_0^{s, p}(\Omega)}
\end{align*}
This implies that $u$ is a nontrivial weak solution of the problem
$P_1$. We thus have proved the following theorem.
\begin{theorem}
Suppose that to each $f\in L^{p'}(\Omega)$, the problem
\begin{align*}
\begin{split}
P_2:~~-\mathscr{L}_\Phi u & = \lambda |u|^{q-2}u +f,\,\,\,f\in L^{p'}(\Omega),\\
u & = 0\,\, \mbox{on}\,\, \mathbb{R}^N\setminus \Omega
\end{split}
\end{align*}
has a nontrivial solution in $\mathfrak{M} \subset W_0^{s,
p}(\Omega)$ for some $\lambda >0$, where $q\in(p, p_s^*)$. Then the
problem
\begin{align*}
\begin{split}
P_1:~~-\mathscr{L}_\Phi u & = \lambda |u|^{q-2}u,\\
u & = 0\,\, \mbox{on}\,\, \mathbb{R}^N\setminus \Omega
\end{split}
\end{align*}
has a nontrivial solution in $W_0^{s, p}(\Omega)$, whenever $\lambda
\in \left( 0, \displaystyle{\inf _{u\neq 0\in W_0^{s,
p}(\Omega)}\frac{||u||_{W_0^{s, p}(\Omega)}}{||u||_q}}\right]$.
\end{theorem}
\section{A necessary condition for the existence of a weak solution}
We now prove a necessary for the existence of a weak solution to the
problem
\begin{eqnarray}
-\mathscr{L}_\Phi u & = & \lambda |u|^{q-2}u+\mu, \nonumber\\
u & = & 0\,\, \mbox{in}\,\, \mathbb{R}^N\setminus
\Omega\label{eq_ns}
\end{eqnarray}
where $\mu$ is a measure.
\begin{theorem}
Suppose $2<p<q$ and $u$ is a weak solution to the problem
$(\ref{eq_ns})$ then $\mu\in L^{1}(\Omega)+W^{-s,p'}(\Omega)$.
\end{theorem}
\begin{proof}
Suppose $u$ is a weak solution to $(\ref{eq_ns})$. Choose a test
function $\phi\in C_0^{\infty}(\Omega)$ such that $\phi(x)\equiv 1$
over a compact subset of $\Omega$ say $K$ and $0\leq \phi \leq 1$ in
$\Omega$ from which one has $\phi\geq \chi_{K}$. Thus
\begin{eqnarray}
\mu(K) & \leq & |\int_{\Omega}-\mathscr{L}_\Phi u \phi
dx+\int_{\Omega}|u|^{q-2}u\phi dx| \nonumber\\
& \leq &
(C_6||-\mathscr{L}_{\Phi}u||_{W^{-s,p'}(\Omega)}+C_7||u||_{q}^{q-1})||\phi||_{s,q}\label{cap}.
\end{eqnarray}
\begin{definition}
Suppose $K\subset\Omega$ is a compact set, then we define the
capacity as
\begin{eqnarray}
Cap_{s,q}(K)&=&\inf\{||\phi||_{W^{s,q}(\Omega)}^q:\phi\in
C_c^{\infty}(\Omega), 0\leq \phi \leq 1,\phi\equiv
1~\text{in}~K\}.\nonumber
\end{eqnarray}
\end{definition}
\noindent We refer to a result from Gallouet et al \cite{gall} that
$\mu\in L^1(\Omega)+W^{-s,p}(\Omega)$ iff $\mu(K)=0$ whenever
$Cap_{s,q}(K)=0$ for $K$ compact subset of $\Omega$. Coming back to
our theorem, if you suppose $Cap_{s,q}(K)=0$ then there exists
$(\phi_n)$ such that $||\phi_n||_{W^{s,q}(\Omega)}^{q}\rightarrow
0$. Hence if this sequence $(\phi_n)$ is used in $(\ref{cap})$, one
has $\mu(K)=0$. Thus we have $\mu\in L^1(\Omega)+W^{-s,p'}(\Omega)$.
%one has $u\in L^{p'}(\Omega)\hookrightarrow L^{1}(\Omega)$
%and hence $-\mathscr{L}_\Phi u\in W^{-s,p'}(\Omega)$.
%On the converse side, given $\mu=f+g\in
%L^1(\Omega)+W^{-s,p'}(\Omega)$, we guarantee the existence of a
%solution to the problem
%\begin{eqnarray}
%-\mathscr{L}_\Phi u &=& f,~\text{in}~\Omega\nonumber\\
%u&=& 0,~\text{in}~\mathbb{R}^N\setminus\Omega,
%\end{eqnarray}
%due to Kussi \cite{kussi}. We name the solution as $u_0\in
%L^{p}(\Omega)\hookrightarrow L^1(\Omega)$. Consider the problem
%\begin{eqnarray}
%-\Delta u &=& \lambda |u_0|^{p-2}u_0,~\text{in}~\Omega\nonumber\\
%u&=& 0,~\text{in}~\mathbb{R}^N\setminus\Omega.
%\end{eqnarray}
%By a classical result there exists $V\in L^{p'}(\Omega)$. Thus
%$-\Delta V\in W^{-s,p'}(\Omega)$. Hence define a measure as
%$\mu^*=f-\Delta V$ for which a weak solution to $(\ref{eq_ns})$
%exists.

\end{proof}
\section*{Acknowledgement}
One of the author R. Kr. Giri thanks the financial assistantship
received from the Ministry of Human Resource Development (M. H. R.
D), Govt. of India. All the authors also acknowledge the facilities
received from the Department of mathematics, National Institute of
Technology Rourkela.

{\sc Ratan Kr Giri}, {\sc D. Choudhuri} and {\sc Amita Soni}\\
Department of Mathematics,\\
National Institute of Technology Rourkela, Rourkela - 769008,
India\\
e-mails: giri90ratan@gmail.com, dc.iit12@gmail.com
\end{document}